 \documentclass[11pt,twoside]{amsart}
\usepackage{amsmath,amsfonts,amscd,amsthm}
\usepackage{amssymb}
\usepackage{color}

\newtheorem{thm}{Theorem}[section]
\newtheorem{prop}[thm]{Proposition}

\newtheorem{cor}[thm]{Corollary}

\theoremstyle{definition}
\newtheorem{defin}[thm]{Definition}

\theoremstyle{remark}

\numberwithin{equation}{section}

\providecommand\ufootnote[1]{{\let\thefootnote\relax\footnote[0]{#1}}}

\newcommand{\cc}{\mathcal C}

\newcommand{\oc}{\mathcal O}

\newcommand{\nb}{\mathbb N}

\newcommand{\ol}{\overline}

\newcommand{\pa}{\partial}
\newcommand{\opa}{\ol\partial}
\newcommand{\wt}{\widetilde}

\begin{document}

\title{Holomorphic approximation and mixed boundary value problems  for $\overline\partial$}


\thanks{The first author would like to thank the university of Notre Dame for its support during her stay in April 2019.  The second author was partially supported by  National Science Foundation grant  DMS-1700003}

\author{Christine Laurent-Thi\'ebaut}
\address{Universit\'e Grenoble-Alpes, Institut Fourier, CS 40700
38058 Grenoble cedex 9, France and CNRS UMR 5582, Institut Fourier,
Saint-Martin d'H\`eres, F-38402, France}
\email{christine.laurent@univ-grenoble-alpes.fr}
\author{Mei-Chi Shaw}
\address{Department of Mathematics,  University of Notre Dame, Notre Dame, IN 46656, USA. }
\email{mei-chi.shaw.1@nd.edu }

 \maketitle

 	\begin{abstract}  In this paper, we  study  holomorphic approximation  using   boundary value problems for $\opa$ on an annulus in the Hilbert space setting. The associated boundary conditions for $\opa$  are the mixed  boundary problems on an annulus. We     characterize  pseudoconvexity and    Runge type property of the domain
by 	 the vanishing of   related $L^2$ cohomology groups.
	
	\end{abstract}	\maketitle
 	\bigskip

\bibliographystyle{amsplain}


Holomorphic approximation theory plays an important role in function theory in one and several complex variables. In one complex variable, the classical Runge approximation theorem is related to   solving the
$\opa$ equation with compact support (see  e.g. Theorem 1.3.1 in H\"ormander's book \cite{Ho1}).  In several complex variables,  it is shown in \cite{LaShrunge}  that  holomorphic  approximation can also
be formulated in terms of Dolbeault cohomology groups. We refer the reader to the   recent paper \cite{FFW} for  a comprehensive and up-to-date  account of this
rich subject.

The purpose of this paper is to associate  holomorphic approximation   to a mixed boundary value problem for $\opa$ on an annulus in the $L^2$ setting.
Let $\Omega_1$ and $\Omega_2$ be two relatively compact domains in a complex hermitian manifold $X$ of complex dimension $n$ such that $\Omega_2\subset\subset\Omega_1$. Consider the annulus $\Omega=\Omega_1\setminus\ol\Omega_2$ between $\Omega_1$ and $\Omega_2$.  Let   $\opa:L^2_{p,q}(\Omega)\to L^2_{p,q+1}(\Omega)$  denote the maximal closure of  $\opa$  in  the weak sense (as defined by H\"ormander in \cite{Ho1}).  By this we mean that
$f\in \text{Dom}(\opa)$ if and only if $f\in L^2_{p,q}(\Omega)$ and $\opa f\in L^2_{p,q+1}(\Omega)$ in the weak sense. It is obvious that
$C^\infty_{p,q}(\ol \Omega)\subset \text{Dom}(\opa)$. If the boundary of $\Omega$ is Lipschitz, the space  $C^\infty_{p,q}(\ol \Omega)$ is dense
in the graph norm of $\opa$ by the Friedrichs  lemma  (see \cite{Ho1} or Lemma 4.3.2 in \cite{ChSh}).

Let  $\opa_c :L^2_{p,q}(\Omega)\to L^2_{p,q+1}(\Omega)$ be   the (strong) minimal closure  of the differential operator $\opa$ in the sense that
$f\in \text{Dom}(\opa_c)$ if and only if $f\in L^2_{p,q}(\Omega)$ and there exists a sequence of forms $f_\nu\in \mathcal D_{p,q}(\Omega)$ such that $f_\nu\to   f$ strongly in $L^2_{p,q}(\Omega)$ and $\opa f_\nu \to \opa f$ strongly in $L^2_{p,q+1}(\Omega)$.  
The two operators $\opa$ and $\opa_c$  are naturally dual to each other (see \cite{ChaSh}).
The  $\overline\partial$-Neumann problem on a domain  arises naturally and is of fundamental importance 
 in several complex variables (see \cite{Ho1,Ho2}, \cite{FoKo}  or \cite{ChSh}).

  The $\opa$-Neumann problem    on an annulus between two pseudoconvex domains in $\mathbb{C}^n$  has been studied  earlier    (see  \cite{Sh1}, \cite{Sh2},    \cite{Ho3}  and \cite{ChaSh}).
Recently, Li and   Shaw \cite{LiSh} introduced  the following mixed boundary problem for $\opa$ on the annulus $\Omega$.   It  was then extended by  Chakrabarti and Harrington in   \cite{ChaHa} where, in particular, they weaken  the regularity condition on the inner  boundary of the annulus from the earlier work in  \cite{Sh1} and \cite{LiSh}.
 In the  $L^2$ setting,  the  $\opa_{\rm mix}$ operator on the annulus is  the   closed realization  of $\opa$ which  satisfies the $\opa$-Neumann boundary condition on the outer boundary $b\Omega_1$ and the $\opa$-Cauchy condition  on the inner boundary $b\Omega_2$.
For $0\leq p, q\leq n$ and $u\in L^2_{p,q}(\Omega)$, $u\in{\rm Dom}(\opa_{\rm mix})$ if and only if there exists $v\in L^2_{p,q+1}(\Omega)$ and a sequence $(u_\nu)_{\nu\in\nb}\subset L^2_{p,q}(\Omega)$ which vanish near $\pa\Omega_2$ such that $u_\nu~\to~u$ in $L^2_{p,q}(\Omega)$ and $\opa u_\nu~\to~v$ in $L^2_{p,q+1}(\Omega)$. If $u\in{\rm Dom}(\opa_{\rm mix})$,  then we define  $\opa_{\rm mix} u=v$.  It is obvious that $\opa_{\rm mix}$
is a densely defined closed operator from one Hilbert space to another and
 $$\opa_c\subseteq \opa_{\rm mix}\subseteq\opa.$$

Let $D$ be a domain in $X$ and  $\oc(D)$ denote the space of holomorphic functions in $D$ and $W^1(D)$ be the Sobolev 1-space on $D$.
The following theorem  is  proved in  Theorems 2.2 and 2.4 in \cite{LiSh}.

\begin{thm}\label{mix}
Assume $X$ is Stein and  both $\Omega_1$ and $\Omega_2$ are pseudoconvex with $C^{1,1}$ boundary then, for any $2\leq q\leq n$ and $q=0$, $H^{0,q}_{\opa_{\rm mix}}(\Omega)=0$.  When $q=1$,   there exists a continuous bijection
\begin{equation}\label{eq:1 forms}\oc(\Omega_2)\cap W^1(\Omega_2)/\oc(\Omega_1)\cap L^2(\Omega_1)\to H^{0,1}_{\opa_{\rm mix}}(\Omega).\end{equation}
 
\end{thm}

 Moreover,  $H^{0,1}_{\opa_{\rm mix}}(\Omega)$ is infinite dimensional (see \cite{LiSh}). In fact, it is  even non-Hausdorff  (see   section 5 in \cite{ChaHa}).  The non-Hausdorff property of the quotient group  is equivalent to    that  the space
 $\oc(\Omega_1)\cap L^2(\Omega_1)$ is not a closed subspace in $\oc(\Omega_2)\cap W^1(\Omega_2)$
 under the $W^1(\Omega_2)$ norm
 (see  Proposition 4.5  in  \cite{Tre}).

 Instead of considering the non-Hausdorff cohomology group $H^{0,1}_{\opa_{\rm mix}}(\Omega)$, we consider the associated  Hausdorff cohomology  group  $^\sigma {( H^{p,0}_{W^1}(D)/H^{p,0}(X))}$ defined by
$$
^\sigma {( H^{p,0}_{W^1}(D)/H^{p,0}(X))}= H^{p,0}_{W^1}(D)/\ol{H^{p,0}(X)},$$
where  $\ol{H^{p,0}(X)}$  is the closure of the space $H^{p,0}(X)$ under the $W^1(D)$-norm.
 It follows from Proposition 1 in \cite{Ca}, that there exists a continuous surjective map

\begin{equation}\label{eq:1 forms sigma} \oc(\Omega_2)\cap W^1(\Omega_2)/\overline{\oc(\Omega_1)\cap L^2(\Omega_1)}\to ~^\sigma {H}^{0,1}_{\opa_{\rm mix}}(\Omega).\end{equation}

From \eqref{eq:1 forms sigma},
if  the space $\oc(\Omega_1)\cap L^2(\Omega_1)$ is dense in $\oc(\Omega_2)\cap W^1(\Omega_2)$ for the $W^1$ topology on $\Omega_2$ then  $^\sigma {H}^{0,1}_{\opa_{\rm mix}}(\Omega)=0$. 
 Thus  the associated Hausdorff cohomology group $^\sigma  {H}^{0,1}_{\opa_{\rm mix}}(\Omega)$ is directly related to holomorphic approximation. This simple observation motivates the present paper.  However,  the $L^2$  condition on the holomorphic functions near the boundary of $\Omega_1$ is of no interest in holomorphic approximation.  We avoid the  growth condition and reformulate another      $\opa$ problem with  mixed boundary  condition which is   more suitable for    holomorphic approximation.

We consider the more general situation: let   $D$ be a relatively compact domain  in a complex hermitian manifold $X$.
  For $0\leq p, q\leq n$, we define a new operator $\opa_{\rm Mix}$ on  ${(L^2_{loc})}^{p,q}(X\setminus\ol D)$, whose domain is the set of all $u\in {(L^2_{loc})}^{p,q}(X)$ such that $u$ is vanishing on $D$ and $\opa u\in {(L^2_{loc})}^{p,q+1}(X)$, where $\opa u$ is taken in the sense of currents. Then we set $\opa_{\rm Mix} f=\opa f$ in the sense of currents. Compared to the $\opa_{\rm mix}$ operator, we do not assume any growth condition at infinity
  of $X$.

The plan of the paper is as follows: In the  first section,
 we formulate a  new mixed boundary condition of $\opa$, denoted by  $\opa_{\rm Mix}$, which is  associated naturally with holomorphic approximation.
 We prove a theorem (see Theorem \ref{Mix}) analogous to Theorem \ref{mix}.

In the second section, we introduce the transposed operator $^t\opa_{\rm Mix}$ to $\opa_{\rm Mix}$ defined on ${(L^2_{loc})}^{n-p,n-q-1}(X\setminus\ol D)$, whose domain is the $u\in L^2_{n-p,n-q-1}(X\setminus\ol D)$ and $u$ is vanishing outside a compact subset of $X$ such that   $\opa u\in L^2_{n-p,n-q}(X\setminus\ol D)$, where $\opa u$ is taken in the sense of currents.
 We  prove the following characterization of approximation of $\opa$-closed forms using a version of  Serre duality.

\begin{thm}
Let $X$ be a Stein manifold of complex dimension $n\geq 2$, $D\subset\subset X$ a relatively compact pseudoconvex domain in $X$ with Lipschitz boundary. Let    $q$ be  a fixed integer such that $0\leq q\leq n-1$. Then, for any $0\leq p\leq n$,
the following assertions are equivalent.
\begin{enumerate}
\item The space of $W^1_{loc}$ $\opa$-closed $(p,q)$-forms on $X$ is dense in the space of $W^1$ $\opa$-closed $(p,q)$-forms on $D$ for the $W^1$ topology on $D$;

\item The natural map
$H^{n-p,n-q}_{\ol D,W^{-1}}(X)\to H^{n-p,n-q}_c(X)$ is injective;

\item ${H}^{n-p,n-q-1}_{^t\opa_{\rm Mix}}(X\setminus\ol D)=0$.
\end{enumerate}

\end{thm}

 Finally, we obtain the following characterization of a pseudoconvex domain satisfying some Runge type property (see Corollary \ref{car}).

\begin{thm}\label{th:Runge Mix}
Let $X$ be a Stein manifold of complex dimension $n\geq 2$  and $D\subset\subset X$ a relatively compact domain in $X$ with $\cc^{1,1}$ boundary such that $X\setminus D$ is connected. Then
the following assertions are equivalent:

\begin{enumerate}
\item  the domain $D$ is pseudoconvex and the space $\oc(X)$ is dense in the space $\oc(D)\cap W^1(D)$  for the $W^1$ topology on $D$;

\item  $H^{n,r}_{\ol D, W^{-1}}(X)=0$, for $2\leq r\leq n-1$, and the natural map
$H^{n,n}_{\ol D,W^{-1}}(X)\to H^{n,n}_c(X)$ is injective;

\item ${H}^{n,q}_{^t\opa_{\rm Mix}}(X\setminus\ol D)=0$, for all $1\leq q\leq n-1$.
\end{enumerate}
\end{thm}

From (1) and (3) in Theorem \ref{th:Runge Mix}, we see that   the vanishing of the cohomology groups ${H}^{n,q}_{^t\opa_{\rm Mix}}(X\setminus\ol D)$  for all $1\leq q\leq n-1$ characterizes pseudoconvexity  and a Runge  type property of $D$.
 This is in contrast to earlier results using cohomology groups on $X\setminus \ol D$ to characterize holomorphic convexity (see  Trapani \cite{Tra}).  It is proved in
  \cite{Tra} that    the vanishing of  the Dolbeault cohomology groups ${H}^{n,q}(X\setminus\ol D)$ for $1\le q\le n-2$ and the Hausdorff property for $q=n-1$ characterizes the holomorphic convexity of $\ol D$. More recently,
  it is proved in    Fu-Laurent-Shaw \cite{FuLaSh} that     the vanishing of  the $L^2$ Dolbeault cohomology groups ${H}^{n,q}_{L^2}(X\setminus\ol D)$ for $1\le q\le n-2$ and the Hausdorff property for $q=n-1$ characterizes   pseudoconvexity of $D$ (see \cite{FuLaSh}).  Thus different cohomology groups  characterize different  holomorphic properties of the domain $D$.   Our results show that   $\opa_{\rm Mix}$   and its transpose $^t\opa_{\rm Mix}$ are naturally associated with holomorphic approximation.

\section{$W^1$-Mergelyan domains and $L^2$ theory for   $\opa$ with   mixed boundary conditions}

Let $X$ be a complex hermitian manifold of complex dimension $n$, where $n\ge 2$.

\begin{defin}
 A relatively compact domain $D$ with Lipschitz boundary in a complex manifold $X$ is called $W^1$-\emph{Mergelyan in $X$} if and only if $\oc(X)$ of holomorphic fuctions in $X$ is dense in the space $\oc_{W^1}(D)$ of  $W^1$ holomorphic functions in $D$ for the $W^1$ topology on $D$. 
\end{defin}

We would like to characterize domains which are $W^1$-Mergelyan in $X$ by means of some adapted mixed boundary value problem for the $\opa$-operator.
Let $L^2_{loc}(X)$ be the space of  $L^2_{loc}$ functions in $X$ endowed with the Fr\'echet topology of $L^2$ convergence on compact subsets, and  $L^2_c(X)$ the space of $L^2$ functions with compact support in $X$ with the inductive limit topology.  These two spaces are dual of each other (see \cite{Se} or \cite{LaLe}).  We use $(L^2_c)^{p,q}(X)$ to denote the space of $(p,q)$-forms with $L^2_{c}(X)$ coefficients.
For $0\leq p, q\leq n$,  we define  the densely defined operator $\opa_K$ from $(L^2_c)^{p,q}(X)$ into $(L^2_c)^{p,q+1}(X)$, whose domain is the space of all $f\in (L^2_c)^{p,q}(X)$ with $\opa f\in (L^2_c)^{p,q+1}(X)$, such that for any $f\in {\rm Dom}(\opa_K)$, $\opa_K f=\opa f$ in the sense of currents. We denote by $\opa_{loc}$ the densely defined transposed operator of $\opa_K$, then $\opa_{loc}$ maps  $(L^2_{loc})^{n-p,n-q-1}(X)$ into $(L^2_{loc})^{n-p,n-q}(X)$ and the domain of $\opa_{loc}$ is the space of all $f\in (L^2_{loc})^{n-p,n-q-1}(X)$ such that $\opa f\in (L^2_{loc})^{n-p,n-q}(X)$.

Let $D$ be a relatively compact domain with Lipschitz boundary in a complex manifold $X$. We are interested in the study in the $L^2$ setting of some operators $\opa_{\rm Mix}$ on $X\setminus\ol D$ such that $\opa_K\subseteq \opa_{\rm Mix}\subseteq\opa_{loc}$, where $\opa_K$ and $\opa_{loc}$ are the previously defined  operators. The domain of $\opa_{\rm Mix}$ is defined as follows:

For $0\leq p, q\leq n$ and $u\in {(L^2_{loc})}^{p,q}(X\setminus\ol D)$, $u\in{\rm Dom}(\opa_{\rm Mix})$ if and only if $u\in {(L^2_{loc})}^{p,q}(X)$,  $u$ is vanishing on $D$ and   $\opa u\in {(L^2_{loc})}^{p,q+1}(X)$, where $\opa u$ is taken in the sense of currents. Then we set $\opa_{\rm Mix} f=\opa f$ in the sense of currents. The transposed operator $^t\opa_{\rm Mix}$ is then an operator  whose domain is given by the set of all $u\in {(L^2_{loc})}^{n-p,n-q-1}(X\setminus\ol D)$, $u\in L^2_{n-p,n-q-1}(X\setminus\ol D)$, $\opa u\in L^2_{n-p,n-q}(X\setminus\ol D)$, where $\opa u$ is taken in the sense of currents, and $u$ is vanishing outside a compact subset of $X$.

For any $0\leq p\leq n$, we get two new differential complexes $({(L^2_{loc})}^{p,\bullet}(X\setminus\ol D),\opa_{\rm Mix})$ and $({(L^2_{loc})}^{n-p,\bullet}(X\setminus\ol D),^t\opa_{\rm Mix})$, which are dual complexes since the boundary of $D$ is Lipschitz (see \cite{LS}).  We denote by $H^{p,q}_{\opa_{\rm Mix}}(X\setminus\ol D)$ and $H^{p,q}_{^t\opa_{\rm Mix}}(X\setminus\ol D)$, $0\leq q\leq n$, the cohomology groups of the complexes $(L^2_{p,\bullet}(X\setminus\ol D),\opa_{\rm Mix})$ and $(L^2_{n-p,\bullet}(X\setminus\ol D),^t\opa_{\rm Mix})$ respectively. We endow the cohomology groups with quotient topology. Then it follows from Serre duality \cite{Se}  that $H^{p,q}_{\opa_{\rm Mix}}(X\setminus\ol D)$ is Hausdorff if and only if $H^{n-p,n-q+1}_{^t\opa_{\rm Mix}}(X\setminus\ol D)$ is Hausdorff. Moreover,  if $H^{p,q}_{^t\opa_{\rm Mix}}(X\setminus\ol D)$
is Hausdorff,  then $H^{p,q}_{^t\opa_{\rm Mix}}(X\setminus\ol D)$ is the dual space of $^\sigma {H}^{n-p,n-q}_{\opa_{\rm Mix}}(X\setminus\ol D)$ the   Hausdorff group associated to $H^{n-p,n-q}_{\opa_{\rm Mix}}(X\setminus\ol D)$.

\begin{thm}\label{Mix}
Let  $X$ be  a Stein manifold of complex dimension $n\geq 2$ with a hermitian metric  and $D$ a relatively compact pseudoconvex domain with $C^{1,1}$ boundary in $X$. Then, for any $0\leq p\leq n$,  we have
\begin{enumerate}
\item  $H^{p,q}_{\opa_{\rm Mix}}(X\setminus\ol D)=0$, if $2\leq q\leq n$ or $q=0$.

 \item There exists a linear continuous bijection
\begin{equation}\label{eq:Iso}l~:~H^{p,0}_{W^1}(D)/H^{p,0}(X)~\to~H^{p,1}_{\opa_{\rm Mix}}(X\setminus\ol D) .\end{equation}
\end{enumerate}
\end{thm}
\begin{proof}
The proof  is similar to  the proof of Theorems 2.2 and 2.4 in \cite{LiSh}.
If $q=0$, $H^{p,0}_{\opa_{\rm Mix}}(X\setminus\ol D)$ is   the space of holomorphic $(p,0)$-forms in $X$, which vanish identically on $D$. Since $X$ is Stein, hence connected, by analytic continuation we get $H^{p,0}_{\opa_{\rm Mix}}(X\setminus\ol D)=0$.

We now assume that $2\leq  q\leq n$.  Let $f\in \ker(\opa_{\rm Mix})\cap{\rm Dom}(\opa_{\rm Mix)}$. Then $f\in {(L^2_{loc})}^{p,q}(X)$, $f=0$ in $D$ and $\opa f=0$ in $X$. Since $X$ is Stein, $H^{p,q}(X)=0$ and by the Dolbeault isomorphism and the interior regularity of the $\opa$, we get $H_{L^2_{loc}}^{p,q}(X)=0$. More precisely there exists $v\in {(W^1_{loc})}^{p,q-1}(X)$ such that $\opa v=f$. Moreover we have $\opa v=0$ on $D$.

Since $q>1$ and $D$ is a relatively compact pseudoconvex domain with $\cc^{1,1}$ boundary, it follows from
\cite{Har} or Theorem 2.2 in \cite{ChaHa}  (see also \cite{Ko} for smooth boundary) that there exists $w\in W^1_{p,q-2}(D)$ such that $\opa w=v$ in $D$. Let $\wt w$ be a $W^1_{loc}$ extension of $w$ to $X$. We set $u=v-\opa\wt w$. Then $u$ is in ${(L^2_{loc})}^{p,q-1}(X)$, u vanishes on $D$ and satisfies $\opa u=f$. This proves (1).

We now consider the case when $q=1$. For any $f\in H^{p,0}_{W^1}(D)$, we extend $f$ as a $W^1_{loc}$ $(p,0)$-form $\wt f=E(f)$ on $X$, where $E$ is a continuous extension operator from $W^1_{p,0}(D)$ into ${(W^1_{loc})}^{p,0}(X)$. This is possible since the boundary of $D$  is $C^{1,1}$.  Then $\opa\wt f\in {(L^2_{loc})}^{p,1}(X)$ and
$\opa\wt f=0$ on $D$. Thus $\opa_{\rm Mix}(\opa\wt f)=0$ in $X\setminus\ol D$. We define a map
\begin{equation}\label{eq:l}l~:~H^{p,0}_{W^1}(D)~\to~H^{p,1}_{\opa_{\rm Mix}}(X\setminus\ol D)\end{equation}
by $l(f)=[\opa\wt f]$.

First, we show that $l$ is well-defined. If $\wt f_1$ is another $W^1_{loc}$ extension of $f$ to $X$, then
$$\opa\wt f-\opa\wt f_1=\opa(\wt f-\wt f_1).$$
Since $\wt f=\wt f_1=f$ on $D$, $\wt f-\wt f_1$ vanishes on $D$ and $\opa\wt f-\opa\wt f_1=\opa_{\rm Mix}(\wt f-\wt f_1)$, that is
$$[\opa\wt f]=[\opa\wt f_1]\quad {\rm in}\quad H^{p,1}_{\opa_{\rm Mix}}(X\setminus\ol D).$$
Thus the map $l$ is well-defined  and it is  continuous if $H^{p,1}_{\opa_{\rm Mix}}(X\setminus\ol D)$ is endowed with the quotient topology.

We will show that the kernel of the map $l$ is $H^{p,0}(X)$. Let $f\in H^{p,0}_{W^1}(D)$ such that $l(f)=[0]$. First we extend $f$ as a $W^1_{loc}$
$(p,0)$-form on $X$. Thus we have that $\opa\wt f$ is a $\opa_{\rm Mix}$-closed form and, since $l(f)=[0]$, it is $\opa_{\rm Mix}$-exact. Therefore there exists $g\in {(L^2_{loc})}^{p,0}(X)$ such that $g=0$ on $D$ and $\opa_{\rm Mix} g=\opa\wt f$. Let $F=\wt f-g$. Then $F$ is holomorphic in $X$ and $F=f$ on $D$. Thus $l(f)=0$ implies that $f$ can be extended as a holomorphic $(p,0)$-form in $X$.

Next we prove that $l$ is surjective. Let $f\in {(L^2_{loc})}^{p,1}(X)\cap\ker(\opa_{\rm Mix})$, then $f=0$ in $D$ and $\opa f=0$ in $X$. Since $X$ is a Stein manifold, using Dolbeault isomorphism and the interior regularity of the $\opa$ operator, there exists a $(p,0)$-form $u\in ({W^1_{loc}})^{p,0}(X)$ such that $\opa u=f$ in $X$. Moreover $u_{|_D}$ is a $W^1$ holomorphic $(p,0)$-form in $D$. Hence $l(u_{|_D})=[\opa u]=[f]$.

 Thus the map defined by \eqref{eq:l}  induces a map
$$l~:~H^{p,0}_{W^1}(D)/H^{p,0}(X)~\to~H^{p,1}_{\opa_{\rm Mix}}(X\setminus\ol D) $$
which is one-to-one continuous and onto, if we endow the quotient space $H^{p,0}_{W^1}(D)/H^{p,0}(X)$ with the quotient topology.

\end{proof}

Using the same arguments as in \cite{LiSh}, one can show that  $H^{0,1}_{\opa_{\rm Mix}}(X\setminus\ol D)$ is infinite dimensional. In fact,  one has the following results using arguments in \cite{ChaHa}.
 \begin{cor} The space    $H^{0,1}_{\opa_{\rm Mix}}(X\setminus\ol D)$ is    non-Hausdorff.

 \end{cor}
 \begin{proof}
 We first show that $H^{p,0}_{W^1}(D)/H^{p,0}(X)$ is non-Hausdorff.
 The non-Hausdorff property of the quotient  space $H^{p,0}_{W^1}(D)/H^{p,0}(X)$ is equivalent to  that   the space
 $H^{p,0}(X)$ is not a closed subspace in $H^{p,0}_{W^1}(D)$
 (see Proposition 4.5 in \cite{Tre}).  The proof of this is exactly the same as in   \cite{ChaHa} and we repeat the arguments
 for the benefit of the reader.

 Let $R:H^{p,0}(X)\to H^{p,0}_{W^1}(D)$ be the restriction map.
  From the Montel theorem,   $R$ is a compact operator.
 Suppose
 $R(H^{p,0}(X))$ is  a closed subspace in $H^{p,0}_{W^1}(D)$. It follows from  the open mapping theorem, the unit ball in  $R(H^{p,0}(X))$ is relatively compact and hence $R(H^{p,0}(X))$
 is finite dimensional. This is a contradiction since $X$ is Stein. Thus $R(H^{p,0}(X))$ is not closed and $H^{p,0}_{W^1}(D)/H^{p,0}(X)$ is non-Hausdorff.

If $H^{0,1}_{\opa_{\rm Mix}}(X\setminus\ol D)$ is Hausdorff, then from   \eqref{eq:Iso} and  the open mapping theorem,  the space   $H^{0,1}_{\opa_{\rm Mix}}(X\setminus\ol D)$
is topologically  isomorphic   to $H^{p,0}_{W^1}(D)/H^{p,0}(X)$, which is non-Hausdorff.  This is a contradiction.
We conclude    that  $H^{0,1}_{\opa_{\rm Mix}}(X\setminus\ol D)$ is also   non-Hausdorff. The corollary is proved.

 \end{proof}

\begin{defin}  We define  the associated Hausdorff   quotient

 \begin{equation}\label{eq:forms sigma} ^\sigma {( H^{p,0}_{W^1}(D)/H^{p,0}(X))}=H^{p,0}_{W^1}(D)/\ol{H^{p,0}(X)}\end{equation}
where  $\ol{H^{p,0}(X)}$  is the closure of the space $H^{p,0}(X)$ under the $W^1(D)$-norm.

\end{defin}

\begin{cor}
Assume $X$ is a Stein manifold of complex dimension $n\geq 2$ and $D$ a relatively compact pseudoconvex domain with $C^{1,1}$ boundary in $X$.

 Suppose that $D$ is $W^1$-Mergelyan. Then   $H^{n,n-1}_{^t\opa_{\rm Mix}}(X\setminus\ol D)=0$.
\end{cor}
\begin{proof}
From (2) in Theorem \ref{Mix} and \eqref{eq:forms sigma},   there exists a map 
$$^\sigma l~:~  H^{p,0}_{W^1}(D)/\ol {H^{p,0}(X)}\to ~ ^\sigma {H}^{p,1}_{\opa_{\rm Mix}}(X\setminus\ol D).$$
which is continuous and onto (see Proposition 1 in \cite{Ca}). Therefore, if $\ol {H^{p,0}(X)}=H^{p,0}_{W^1}(D)$, then $^\sigma {H}^{p,1}_{\opa_{\rm Mix}}(X\setminus\ol D)=0$.

Thus if $D$ is $W^1$-Mergelyan in $X$, $^\sigma {H}^{0,1}_{\opa_{\rm Mix}}(X\setminus\ol D)=0$.
It follows from   Serre duality and from Theorem \ref{Mix} that $H^{n,n-1}_{^t\opa_{\rm Mix}}(X\setminus\ol D)$ is Hausdorff, since $H^{0,2}_{\opa_{\rm Mix}}(X\setminus\ol D)=0$. Using  again Serre duality,  we get $H^{n,n-1}_{^t\opa_{\rm Mix}}(X\setminus\ol D)=0$.
\end{proof}

\section{The $W^1$ $q$-Mergelyan density property}
Let $X$ be a complex hermitian manifold of complex dimension $n$, where $n\ge 1$.
In this section we extend the approximation results to arbitrary $(p,q)$-forms.
\begin{defin}
A relatively compact domain $D$ with Lipschitz boundary in $X$ is \emph{$W^1$ $(p,q)$-Mergelyan}, for $0\leq p\leq n$ and $0\leq q\leq n-1$, if and only if the space $Z^{p,q}_{W^1_{loc}}(X)$ of $W^1_{loc}$ $\opa$-closed $(p,q)$-forms in $X$ is dense in the space $Z^{p,q}_{W^1}(D)$ of  $W^1$ $\opa$-closed $(p,q)$-forms in $D$ for the $W^1$ topology on $D$.

For $p=q=0$, we will simply say that the domain  is \emph{$W^1$-Mergelyan} in $X$.
\end{defin}

If $D\subset\subset X$ is a relatively compact domain with Lipschitz boundary in $X$, we denote by $H^{r,s}_{\ol D, W^{-1}}(X)$ the  Dolbeault cohomology groups of $W^{-1}$ currents with prescribed support in $\ol D$ and by ${H}^{r,s}_{^t\opa_{\rm Mix}}(X\setminus\ol D)$ the Dolbeault cohomology groups of $L^2$ forms in $X\setminus\ol D$ vanishing outside a compact subset of $X$.
  We have that $W^s(D)$ is a reflexive Banach space, i.e. $(W^{-s}_{\overline D}(X))'=W^s(D)$.

\begin{thm}\label{qW1Mergelyan}
Let $X$ be a non compact complex manifold of complex dimension $n\geq 1$, $D\subset\subset X$ a relatively compact domain with Lipschitz boundary in $X$ and $p$ and $q$ be fixed integers such that $0\leq p\leq n$ and $0\leq q\leq n-1$. Assume  that $H^{n-p,n-q}_c(X)$ and  $H^{n-p,n-q}_{\ol D,W^{-1}}(X)$ are Hausdorff. Then $D$ is a $W^1$ $(p,q)$-Mergelyan domain in $X$ if and only if the natural map
$H^{n-p,n-q}_{\ol D,W^{-1}}(X)\to H^{n-p,n-q}_c(X)$
is injective.
\end{thm}
\begin{proof}
Assume $D$ is $W^1$ $(p,q)$-Mergelyan in $X$ and let $T\in W^{-1}_{n-p,n-q}(X)$ with support contained in $\ol D$ such that $\opa T=0$.  We assume  that the cohomological class $[T]$ of $T$ vanishes in $H^{n-p,n-q}_c(X)$, which means that there exists $S\in W^{-1}_{n-p,n-q-1}(X)$ with compact support in $X$ such that $T=\opa S$. Since $H^{n-p,n-q}_{\ol D,W^{-1}}(X)$ is Hausdorff, then $[T]=0$ in $H^{n-p,n-q}_{\ol D,W^{-1}}(X)$ if and only if, for any form  $\varphi\in Z^{p,q}_{W^1}(D)$, we have $<T, \varphi>=0$. But, as $D$ is $W^1$ $(p,q)$-Mergelyan in $X$, there exists a sequence $(\varphi_k)_{k\in\nb}$ of $W^1_{loc}$ $\opa$-closed $(p,q)$-forms in $X$  which converge to $\varphi$ in $W^1(D)$. So
$$<T,\varphi>=\lim_{k\to\infty}<T,\varphi_k>=\lim_{k\to\infty}<\opa S,\varphi_k>=\pm \lim_{k\to\infty}<S,\opa\varphi_k>=0.$$

Conversely, by the Hahn-Banach theorem, it is sufficient to prove that, for any form $g\in Z^{p,q}_{W^1}(D)$ and any  $(n-p,n-q)$-current $T$ in $W^{-1}_{n-p,n-q}(X)$ with compact support in $\ol D$ such that $<T,f>=0$ for any form $f\in Z^{p,q}_{W^1_{loc}}(X)$, we have $<T,g>=0$. Since $H^{n-p,n-q}_c(X)$ is Hausdorff, the hypothesis on $T$ implies that there exists a $W^{-1}$ $(n-p,n-q-1)$-current $S$ with compact support in $X$ such that $T=\opa S$. The injectivity of the natural map
$H^{n-p,n-q}_{\ol D,W^{-1}}(X)\to H^{n-p,n-q}_c(X)$ implies that there exists a $W^{-1}$ $(n-p,n-q-1)$-current $U$ with compact support in $\ol D$ such that $T=\opa U$. Hence since the boundary of $D$ is Lipschitz, for any $g\in Z^{p,q}_{W^1}(D)$, we get
$$<T,g>=<\opa U, g>=\pm <U,\opa g>=0.$$
\end{proof}

\begin{prop}\label{qW1complement}
Let $X$ be a non compact complex manifold of complex dimension $n\geq 2$, $D\subset\subset X$ a relatively compact domain in $X$ with Lipschitz boundary and $p$ and $q$ fixed integers such that $0\leq p\leq n$ and $0\leq q\leq n-2$. Assume  that  $H^{n-p,n-q-1}_c(X)=0$. Then ${H}^{n-p,n-q-1}_{^t\opa_{\rm Mix}}(X\setminus\ol D)=0$ if and only if the natural map
$H^{n-p,n-q}_{\ol D,W^{-1}}(X)\to H^{n-p,n-q}_c(X)$
is injective.
\end{prop}
\begin{proof}
We first consider the necessary condition. Let $T\in W^{-1}_{n-p,n-q}(X)$ be a $\opa$-closed current with support contained in $\ol D$ such that the cohomological class $[T]$ of $T$ vanishes in $H^{n-p,n-q}_c(X)$. By the interior regularity property of the $\opa$-operator and the Dolbeault isomorphism, there exists $g\in L^2_{n-p,n-q-1}(X)$ and compactly supported such that $T=\opa g$. Since the support of $T$ is contained in $\ol D$, we have $\opa g=0$ on $X\setminus\ol D$. Therefore the vanishing of the group $H^{n-p,n-q-1}_{^t\opa_{\rm Mix}}(X\setminus\ol D)$ implies that there exists $u\in L^2_{n-p,n-q-2}(X\setminus\ol D)$ vanishing outside a compact subset of $X$ and such that $\opa u=g$ on $X\setminus\ol D$. Since the boundary of $D$ is Lipschitz there exists $\wt u$ a $L^2$ extension of $u$ to $X$, we set $S=g-\opa\wt u$, then $S\in W^{-1}(X)$ satisfies $T=\opa S$ and ${\rm supp}~S\subset\ol D$.

Conversely, let $g$ be a  $\opa$-closed $(n-p,n-q-1)$-form in $L^2_{n-p,n-q-1}(X\setminus\ol D)$ which vanishes outside a compact subset of $X$ and $\wt g$ an $L^2$ extension of $g$ to $X$, then $\wt g$ has compact support in $X$ and $T=\opa\wt g$ is a current in $W^{-1}_{n-p,n-q}(X)$ with support in $\ol D$. By the injectivity of the natural map
$H^{n-p,n-q}_{\ol D,W^{-1}}(X)\to H^{n-p,n-q}_c(X)$, there exists $S\in W^{-1}_{n-p,n-q-1}(X)$ with support contained in $\ol D$ and such that $\opa S=T$. We set $U=\wt g-S$. Then  $U$ is   a $W^{-1}$ $\opa$-closed $(n-p,n-q-1)$-current with compact support in $X$ such that $U_{|_{X\setminus\ol D}}=g$ in $X\setminus\ol D$. Since $H^{n-p,n-q-1}_c(X)=0$, by the interior regularity property of the $\opa$-operator and the Dolbeault isomorphism, we have $U=\opa w$ for some $w\in L^2_{n-p,n-q-2}(X)$ with compact support in $X$. Finally we get $g=U_{|_{X\setminus\ol D}}=\opa (w_{|_{X\setminus\ol D}})$.
\end{proof}

\begin{cor}\label{merg}
Let $X$ be a Stein hermitian  manifold of complex dimension $n\geq 2$ and $D\subset\subset X$ a relatively compact pseudoconvex domain with $\cc^{1,1}$ boundary in $X$. Then the following assertions are equivalent:

i) the domain $D$ is $W^1$-Mergelyan in $X$,

ii) the natural map
$H^{n,n}_{\ol D,W^{-1}}(X)\to H^{n,n}_c(X)$
is injective,

iii) ${H}^{n,n-1}_{^t\opa_{\rm Mix}}(X\setminus\ol D)=0$.
\end{cor}
\begin{proof}
Since $X$ is Stein, we have $H^{n,n-1}_c(X)=0$ and $H^{n,n}_c(X)$ is Hausdorff. The domain $D$ being relatively compact, pseudoconvex with $\cc^{1,1}$ boundary in $X$,  we have $H^{0,1}_{W^1}( D)=0$.  Then Serre duality implies that $H^{n,n}_{\ol D, W^{-1}}(X)$ is Hausdorff.
The corollary follows then from Theorem \ref{qW1Mergelyan} and Proposition \ref{qW1complement}.
\end{proof}

Finally using the characterization of pseudoconvexity by means of $W^1$ cohomology and Serre duality, we can prove the following corollary.

\begin{cor}\label{car}
Let $X$ be a Stein hermitian  manifold of complex dimension $n\geq 2$  and $D\subset\subset X$ a relatively compact domain in $X$ with $\cc^{1,1}$ boundary such that $X\setminus D$ is connected. Then
the following assertions are equivalent:

(i) the domain $D$ is pseudoconvex and $W^1$-Mergelyan in $X$;

(ii) $H^{n,r}_{\ol D, W^{-1}}(X)=0$, for $2\leq r\leq n-1$, and the natural map
$H^{n,n}_{\ol D,W^{-1}}(X)\to H^{n,n}_c(X)$ is injective;

(iii) ${H}^{n,q}_{^t\opa_{\rm Mix}}(X\setminus\ol D)=0$, for all $1\leq q\leq n-1$.
\end{cor}
\begin{proof}
Consider the equivalence between (i) and (ii).
We first notice  that a domain $D$ with $\cc^{1,1}$ boundary is pseudoconvex if and only if $H^{0,q}_{W^1}(D)=0$ for all $1\leq q\leq n-1$.
This follows from \cite{Har} or  Theorem 2.2 in \cite{ChaHa} for the necessary condition and Theorem 5.1 in \cite{FuLaSh} for the sufficient condition.

 Recall that applying Serre duality, we get that $H^{n,n-r+1}_{\ol D,W^{-1}}(X)$ is Hausdorff if and only if $H^{0,r}_{W^1}(D)$ is Hausdorff for each $0\leq r\leq n$ and, when both are Hausdorff, $H^{n,r}_{\ol D,W^{-1}}(X)$ is the dual space of $H^{0,n-r}_{W^1}(D)$.

 Let us prove that (i) implies (ii).
 From the previous remarks we get that if $D$ is  pseudoconvex then $H^{0,q}_{W^1}(D)=0$ for all $1\leq q\leq n-1$ and therefore $H^{n,r}_{\ol D,W^{-1}}(X)=0$ for all $2\leq r\leq n-1$ and $H^{n,n}_{\ol D,W^{-1}}(X)$ is Hausdorff. If moreover $D$ is also $W^1$-Mergelyan in $X$, then the natural map
$H^{n,n}_{\ol D,W^{-1}}(X)\to H^{n,n}_c(X)$ is injective by Corollary \ref{merg}.

Conversely we first prove that the injectivity of the natural map
$H^{n,n}_{\ol D,W^{-1}}(X)\to H^{n,n}_c(X)$  implies that $H^{n,n}_{\ol D,W^{-1}}(X)$ is Hausdorff. Let $T$ be a $W^{-1}$ $(n,n)$-current with support in $\ol D$ such that $<T,\varphi>=0$ for any $W^1$ holomorphic function $\varphi$ on $D$. In particular $<T,\varphi>=0$ for any holomorphic function $\varphi$ on $X$. Since   $X$ is Stein, $H^{n,n}_c(X)$ is Hausdorff and therefore $T=\opa S$ for some $W^{-1}$ $(n,n-1)$-current $S$ with compact support in $X$, i.e. $[T]=0$ in $H^{n,n}_c(X)$. By the injectivity of the map $H^{n,n}_{\ol D,W^{-1}}(X)\to H^{n,n}_c(X)$, we get that $T=\opa U$ for some $W^{-1}$ $(n,n-1)$-current $U$ with support in $\ol D$, which ends the proof.

Now assume (ii) is satisfied. Then $D$ satisfies $H^{n,r}_{\ol D,W^{-1}}(X)=0$ for all $2\leq r\leq n-1$ and $H^{n,n}_{\ol D,W^{-1}}(X)$ is Hausdorff.
Applying Serre duality we get $H^{0,q}_{W^1}(D)=0$ for all $1\leq q\leq n-2$ and $H^{0,n-1}_{W^1}(D)$ is Hausdorff but, as $X$ is Stein and $X\setminus D$ is connected, $H^{0,n-1}_{W^1}(D)=0$ (see section 3 in \cite{LS}).
Therefore $D$ is pseudoconvex by the characterization given at the begining of the proof.
It remains to use Corollary \ref{merg} to get that $D$ is $W^1$-Mergelyan in $X$.

We next prove   the equivalence between (ii) and (iii).
 
Above we proved in particular that if $X$ is Stein and  $X\setminus D$ is connected, then $H^{n,r}_{\ol D,W^{-1}}(X)=0$ for all $2\leq r\leq n-1$ and $H^{n,n}_{\ol D,W^{-1}}(X)$ is Hausdorff if and only if $H^{0,q}_{W^1}(D)=0$ for all $1\leq q\leq n-1$. Recall also that the injectivity of the natural map
$H^{n,n}_{\ol D,W^{-1}}(X)\to H^{n,n}_c(X)$  implies that $H^{n,n}_{\ol D,W^{-1}}(X)$ is Hausdorff.
Therefore assertion (ii) implies $H^{0,q}_{W^1}(D)=0$ for all $1\leq q\leq n-1$ , which is equivalent to ${H}^{n,q}_{L^2}(X\setminus\ol D)=0$ for all $1\leq q\leq n-2$ and ${H}^{n,n-1}_{L^2}(X\setminus\ol D)$ is Hausdorff by Theorem 4.8 in \cite{FuLaSh}.

Since $X$ is Stein, by Proposition \ref{qW1complement}, the injectivity of the natural map
$H^{n,n}_{\ol D,W^{-1}}(X)\to H^{n,n}_c(X)$  implies $H^{n,n-1}_{^t\opa_{\rm Mix}}(X\setminus\ol D)=0$. Therefore to get that (ii) implies (iii),
it remains to prove that for each
$1\leq q\leq n-2$, $H^{n,q}_{L^2}(X\setminus\ol D)=0$ implies $H^{n,q}_{^t\opa_{\rm Mix}}(X\setminus\ol D)=0$.
To get this it is sufficient to prove that the natural map from $H^{n,q}_{^t\opa_{\rm Mix}}(X\setminus\ol D)$ into $H^{n,q}_{L^2}(X\setminus\ol D)$ is injective.
Let $f\in L^2_{n,q}(X\setminus\ol D)$ be a $\opa$-closed form which vanishes outside a compact subset $K$ of $X$. Assume $[f]=0$ in $H^{n,q}_{L^2}(X\setminus\ol D)$, then there exists $g\in L^2_{n,q-1}(X\setminus\ol D)$ such that $f=\opa g$ on $X\setminus\ol D$. Consider a function $\chi$ with compact support in $X$ such that $\chi\equiv 1$ on a neighborhood of $\ol D\cup K$. We set $\wt g= \chi g$. Then $\opa\wt g=\opa\chi\wedge g+\chi\opa g=\opa\chi\wedge g+f$ and the form $\opa\chi\wedge g$ can be extended by $0$ to an $L^2$ $\opa$-closed  $(n,q)$-form with compact support in $X$. Since $X$ is Stein, there is an $h\in (L^2_{loc})^{n,q-1}(X)$ with compact support such that $\opa h=\opa\chi\wedge g$ on $X$ and it follows that $\opa\wt g=\opa h+f$ on $X\setminus\ol D$. Then $u=\wt g-h$ vanishes outside a compact subset of $X$ and $\opa u=f$, which ends the proof of the injectivity.

Now assume (iii) holds, i.e. ${H}^{n,q}_{^t\opa_{\rm Mix}}(X\setminus\ol D)=0$, for all $1\leq q\leq n-1$.
We first prove that, for each
$1\leq q\leq n-2$, $H^{n,q}_{^t\opa_{\rm Mix}}(X\setminus\ol D)=0$ implies $H^{n,q}_{L^2}(X\setminus\ol D)=0$ and that $H^{n,n-1}_{^t\opa_{\rm Mix}}(X\setminus\ol D)=0$, implies $H^{n,n-1}_{L^2}(X\setminus\ol D)$ is Hausdorff.

Since $X$ is a Stein manifold, there exists a relatively compact strictly pseudoconvex domain $U$ in $X$ with $\cc^2$ boundary such that $D\subset\subset U$. As already noticed previously, the properties of $U$ imply that ${H}^{n,q}_{L^2}(X\setminus\ol U)=0$ for all $1\leq q\leq n-2$ and ${H}^{n,n-1}_{L^2}(X\setminus\ol U)$ is Hausdorff.

Let $1\leq q\leq n-2$ and $f\in L^2_{n,q}(X\setminus\ol D)$ a $\opa$-closed form, then there exists $g\in L^2_{n,q-1}(X\setminus\ol U)$ such that $f=\opa g$ on $X\setminus\ol U$. Let $V\subset\subset X$ be a neighborhood of $\ol U$ and $\chi$ a smooth function equal to $1$ on $X\setminus V$ and with support contained in $X\setminus\ol U$.
Therefore the form $f-\opa(\chi g)=(1-\chi)f-\opa\chi\wedge g$ vanishes outside the compact subset $\ol V$ and belongs to the domain of $^t\opa_{\rm Mix}$. So if $H^{n,q}_{^t\opa_{\rm Mix}}(X\setminus\ol D)=0$, then $f-\opa(\chi g)=\opa u$ for some $u\in L^2_{n,q-1}(X\setminus\ol D)$ and $f=\opa(\chi g+u)$, which means $H^{n,q}_{L^2}(X\setminus\ol D)=0$.

Let $f\in L^2_{n,n-1}(X\setminus\ol D)$ a $\opa$-closed form such that $\int_X f\wedge\varphi=0$ for any $\opa$-closed $L^2$ $(0,1)$-form $\varphi$ on $X$ which vanishes on the closure of $D$ and outside a compact subset of $X$ and in particular for any $\opa$-closed $L^2$ $(0,1)$-form $\varphi$ on $X$ which vanishes on the closure of $U$ and outside a compact subset of $X$. Since ${H}^{n,n-1}_{L^2}(X\setminus\ol U)$ is Hausdorff, there exists $g\in L^2_{n,q-1}(X\setminus\ol U)$ such that $f=\opa g$ on $X\setminus\ol U$. Then we can repeat the end of the proof of the previous assertion.

Therefore (iii) implies $H^{0,q}_{W^1}(\ol D)=0$ for all $1\leq q\leq n-1$ (see Theorem 4.8 in \cite{FuLaSh}) and we get $H^{n,r}_{\ol D,W^{-1}}(X)=0$ for all $2\leq r\leq n-1$ by Serre duality.
Finally using Proposition \ref{qW1complement}, we obtain that the natural map
$H^{n,n}_{\ol D,W^{-1}}(X)\to H^{n,n}_c(X)$ is injective, which ends the proof.
 
\end{proof}

From Corollary \ref{car}, the vanishing of the cohomology groups ${H}^{n,q}_{^t\opa_{\rm Mix}}(X\setminus\ol D)$    characterizes pseudoconvexity  and $W^1$-Mergelyan  property of $D$.

\providecommand{\bysame}{\leavevmode\hbox to3em{\hrulefill}\thinspace}
\providecommand{\MR}{\relax\ifhmode\unskip\space\fi MR }
\providecommand{\MRhref}[2]{%
  \href{http://www.ams.org/mathscinet-getitem?mr=#1}{#2}
}
\providecommand{\href}[2]{#2}

\enddocument

\end